\documentclass[a4paper, 12pt]{amsart}
\usepackage[T1]{fontenc}
\usepackage{textcomp }
\usepackage{amsmath}
\usepackage{amsfonts}
\usepackage{hyperref}
\usepackage[noabbrev]{cleveref}
\usepackage{amssymb}
\usepackage{amsthm}
\usepackage{tikz}
\usetikzlibrary{cd}
\usepackage{cancel}
\usepackage{multicol}
\usepackage{geometry} 
\usepackage{parskip}
\usepackage{multicol}
\usepackage{xcolor}
\usepackage{ stmaryrd }
\usepackage[polish,english]{babel}
\usepackage{csquotes}

\newtheorem{thm}{Theorem}[subsection]
\newtheorem{lm}[thm]{Lemma}

\newtheorem{proposition}[thm]{Proposition}
\newtheorem{cor}[thm]{Corollary}
\newtheorem{rk}[thm]{Proposition}
\theoremstyle{definition}
\newtheorem{df}[thm]{Definition}

\newtheorem*{claim*}{Claim}
\renewcommand{\thethm}{%
  \ifnum\value{subsection}=0 
    \thesection
  \else
    \thesubsection
  \fi%
  .\arabic{thm}}

\let\int\relax
\DeclareMathOperator{\int}{int}

\DeclareMathOperator{\add}{add}

\let\epsilon\varepsilon
\let\phi\varphi

\let\le\leqslant
\let\leq\le
\let\ge\geqslant
\let\geq\ge

\usepackage[backend=biber, sorting=ynt]{biblatex}

\title{Star operation, microscopic sets and porous sets}

\author[D. Perkowska]{Daria Perkowska}
\email{daria.perkowska@pwr.edu.pl}

\author[Sz. Żeberski]{Szymon Żeberski}
\email{szymon.zeberski@pwr.edu.pl}


\address{Daria Perkowska, Szymon Żeberski, Faculty of Pure and Applied Mathematics, Wrocław University of Science and Technology, 50-370 Wrocław, Poland}

\date{}
\begin{document}

\maketitle

\begin{abstract}





This paper explores the interplay between star operations, microscopic sets, and porous sets. The study focuses on the Galvin-Mycielski-Solovay theorem, which characterizes strongly measure zero sets and their interactions with meager sets. Results include the investigation of the star operation \(\mathcal{F}^*\) and its properties. The paper also examines the relationship between porous sets and microscopic sets. Additionally, the work presents constructions of families \(\mathcal{F}\) in \(\mathcal{P}(\mathbb{Z})\), \(\mathcal{P}(\mathbb{Z}^\omega)\), and \(\mathcal{P}(2^\omega)\) that satisfy \(\mathcal{F} = \mathcal{F}^*\). Theorems and lemmas are provided to establish conditions under which \(\mathcal{F}^{**} = \mathcal{F}\) and to analyze the implications of the Borel Conjecture and its dual. The paper concludes with a discussion of microscopic sets and their properties, including their interactions with porous sets and the non-equivalence of certain classes of sets.

\end{abstract}

\section{Introduction and notation}

 We will use standard notation following from \cite{BBB}.

\begin{df}
We say that a subset $\mathcal{I} \subseteq P(X)$ is an ideal if the following properties are satisfied:
\begin{itemize}
    \item $\emptyset \in \mathcal{I}$,
    \item When $A \in \mathcal{I}$ and $B \subseteq X$, then $B \subseteq A \Longrightarrow B \in \mathcal{I}$,
    \item If $A,B \in \mathcal{I}$ then $A \cup B \in \mathcal{I}$.
\end{itemize}
\end{df}

We say that a subset $\mathcal{I} \subseteq P(X)$ is a $\sigma$-ideal if it is an ideal and it is closed under countable unions.
    We say that an ideal $\mathcal{I}\subseteq\mathcal{P}(X) $ is proper if $X\notin\mathcal{I}$. In this paper we will only consider proper ideals.

We define cardinal invariants of an ideal $\mathcal{I}$ by: 

$$add(\mathcal{I})=\min \{|\mathcal{A}|:\ \mathcal{A} \subseteq \mathcal{I} \wedge \bigcup \mathcal{A} \notin \mathcal{I} \},$$
$$cov(\mathcal{I})=\min \{|A|: A \subseteq \mathcal{I} \wedge \bigcup A =X \},$$
$$cof(\mathcal{I})=\min \{|\mathcal{B}|: \mathcal{B} \subseteq \mathcal{I} \wedge (\forall A \in \mathcal{I})(\exists B \in \mathcal{B})(A \subseteq B) \}.$$

We will be investigating subsets of the Cantor space $2^\omega$ and families of such subsets. The basis of the Cantor space is formed by clopen sets, denoted by $[\beta]$, where $\beta$ is a finite sequence of $0$ and $1$. 
\[
[\beta]=\{x\in 2^\omega:\ \beta\subseteq x\}.
\]
By $\lambda$ we will denote the Lebesgue measure. We will consider the family of sets of Lebesgue's measure zero, which we shall denote by $\mathcal{N}$, and the family of meager sets denoted by $\mathcal{M}$. The $\sigma$-ideal of all countable subsets will be denoted by $Count$. 

\section{Operation *}
The operation * was first defined and investigated in \cite{Seredyski1989SomeOR} by Seredyński. He considered the ideal of bounded sets. 
\begin{df}

     Let $\langle G,+, \varrho\rangle$ be an abelian metric unbounded group with metric $\varrho$.  Let
$$
\mathcal{BD}=\{X \subseteq G: \operatorname{diam}(X)<\infty\}
.$$
\end{df}
In particular, he proved that the ideal of bounded sets $\mathcal{BD}$ satisfies $\mathcal{BD}=\mathcal{BD}^{**}$.

In this article we will only consider abelian groups.
If $(X,+)$ is a group then
for $A, B \in \mathcal{P}(X)$, 
we will write 
$$A+B= \{a+b: a \in A, b \in B\}.$$
\begin{df}
For a family $\mathcal{F}\subseteq \mathcal{P}(X)$, where $X$ is abelian polish group let:
$$\mathcal{F}^{*}=\{A\subseteq X:\ \forall F\in \mathcal{F} \;\; A+F \neq X \}.$$
\end{df}
In any situation, the following 
equivalence occurs useful if $\mathcal{F}$ is reflection invariant and $X$ is abelian.
$$A+F\neq X\; \Longleftrightarrow\;  \exists y\; (A+y)\cap F =\emptyset$$
It was later studied among others by Pawlikowski and Sabok who showed in \cite{Pawlikowski2008-PAWTS} that $Count = Count^{**}$. Horbaczewska and Lindner  in \cite{article} assuming CH investigated conditions under which $\mathcal{I} = \mathcal{I}^{**}$ for certain ideals $\mathcal{I}$. Moreover, Solecki in \cite{solecki} constructed a $\sigma$-ideal $\mathcal{I}$ satisfying $\mathcal{I}^* = Count$.

Here are some well-known facts about operation $^*.$

\begin{thm}
    For any $\mathcal{F,G} \subseteq \mathcal{P}(X)$ we have:
    \begin{itemize}
    \item  $\mathcal{G}\subseteq\mathcal{F}^{*} \Longrightarrow \mathcal{F}\subseteq\mathcal{G}^{*} $,
    \item $\mathcal{F}\subseteq\mathcal{F}^{**} $,
    \item $\mathcal{G}\subseteq\mathcal{F} \Longrightarrow \mathcal{F}^*\subseteq\mathcal{G}^{*} $,
    \item $\mathcal{F}^*=\mathcal{F}^{***}$,
    \item $\mathcal{F}^{*}$ is closed under taking subsets and translation invariant.
\end{itemize}
\end{thm}

\begin{rk}
If $\mathcal{I} \subseteq \mathcal{P}(X)$ is translation invariant $\sigma$-ideal, then $Count \subseteq \mathcal{I}^{*}$.
\end{rk}

\begin{rk}\label{eeee} \cite{article}
$Count^*$ is the union of all proper translation invariant $\sigma$-ideals of subsets of $X$.
\end{rk}

As a consequence of \Cref{eeee} we have 
$$\mathcal{M} \neq Count^*,$$
$$\mathcal{N} \neq Count^*.$$
\begin{rk}
$Count^{*}$ is not an ideal. 
\end{rk}

\begin{proof}
From \Cref{eeee} we know that $\mathcal{M} \subseteq Count^{*}$ and $\mathcal{N} \subseteq Count^{*}$. But there exists $A \in \mathcal{N} $ and $B\in \mathcal{M}$ such that $ 2^{\omega} = A\cup B$. So $Count^{*}$ cannot be an ideal.
\end{proof}

We will try to find conditions under which $\mathcal{F}=\mathcal{F}^{**}$ holds. First, we start with a simple characterization.

\begin{lm}
    $\mathcal{F}=\mathcal{F}^{**}\Longleftrightarrow \exists \mathcal{A}\; \mathcal{F}=\mathcal{A}^{*}.$ 
\end{lm}
\begin{proof}
    $(\Longrightarrow)$ Just take $\mathcal{A}=\mathcal{F}^{*}$.

    $(\Longleftarrow)$ We have that $\mathcal{F}=\mathcal{A}^{*}$. So $\mathcal{F}^{*}=\mathcal{A}^{**}$ and $\mathcal{F}^{**}=\mathcal{A}^{***}=\mathcal{A}^{*}=\mathcal{F}$.
\end{proof}

In \cite{article}, Horbaczewska and Lindner established the following lemma, which provides a useful criterion for determining when an ideal satisfies the equality $\mathcal{I} = \mathcal{I}^{**}$.

\begin{lm}\label{ccc}(Horbaczewska, Lindner) 
For any $\mathcal{F}\subseteq \mathcal{P}(X)$ the following conditions are equivalent:

\begin{itemize}

    \item $(\forall{A\notin \mathcal{F}}) (\mathcal{F} \cup \{A\})^* \neq \mathcal{F}^*$,
    \item $\mathcal{F}=\mathcal{F}^{**} $.
   
\end{itemize}
\end{lm}

Now, applying \Cref{ccc}, we can establish the following theorem. A related version was proved in \cite{article}, though in the setting of $\mathbb{R}$ and under the assumption of the Continuum Hypothesis.

\begin{thm} \label{bbbbb}
    If $\mathcal{J} \subseteq \mathcal{P}(2^{\omega})$ is a translation  and reflection invariant proper $\sigma$-ideal with $cof(\mathcal{J})\leqslant \mathfrak{c}$, $add(\mathcal{J})= \mathfrak{c}$ and any $A \notin \mathcal{J}$, then
$$(\mathcal{J} \cup \{A\})^* \neq \mathcal{J}^*.$$
\end{thm}

\begin{proof}
Fix $A\notin\mathcal{J}$.
Let $\mathcal{F} $ be a basis of $\mathcal{J}$ of size $\mathfrak{c}.$ 

Let $\{z_\alpha\}_{\alpha < \mathfrak{c}}=2^\omega$ be an enumeration of $2^\omega$ and let $\{F_\alpha\}_{\alpha < \mathfrak{c}}$ be an enumeration of all sets from $\mathcal{F}$. We  will build sequences  $\{x_\alpha\}_{\alpha < \mathfrak{c}}$ and $\{r_\alpha\}_{\alpha < \mathfrak{c}}$ satisfying for every $\lambda<\mathfrak c$ the following conditions:
\begin{enumerate}
    \item $r_\lambda \notin \bigcup_{\alpha_1,\alpha_2 < \lambda} (F_{\alpha_{1}} + x_{\alpha_{2}})$,
    \item $z_\lambda\in x_\lambda+A$,
    \item $x_\lambda\notin \bigcup_{\alpha_1,\alpha_2 \leqslant \lambda} (r_{\alpha_{1}} - F_{\alpha_{2}} )$.
    
\end{enumerate}
Start with two different $x_0$ and $r_0$ such that $z_0\in x_0+A$ and $x_0\notin r_0-F_0$. 

Let $\lambda < \mathfrak{c}$. Suppose that we already constructed $\{x_\alpha\}_{\alpha < \lambda}$ and $\{r_\alpha\}_{\alpha < \lambda}$. 
Since $\mathfrak{c}= add(\mathcal{J})\le cov(\mathcal{J})$ the set $\bigcup_{\alpha_1,\alpha_2 < \lambda} (F_{\alpha_{1}} + x_{\alpha_{2}}) \neq 2^{\omega}$. Hence  we can find $r_{\lambda}$ satisfying (1).

Let $B_\lambda = 2^{\omega} \backslash \bigcup_{\alpha_1,\alpha_2 \leqslant \lambda} (r_{\alpha_{1}} - F_{\alpha_{2}} )$. Since $\lambda <\mathfrak c = add(\mathcal{J})$ the set  $2^{\omega} \backslash B_{\lambda}$ belongs to $\mathcal{J}$. Consequently, $2^{\omega} \backslash (z_{\lambda} - B_{\lambda})= z_{\lambda}- (2^{\omega}\backslash B_{\lambda}) \in \mathcal{J}$. Since $A \notin \mathcal{J}$, then $A \nsubseteq 2^{\omega} \backslash (z_\lambda - B_\lambda)$. So, $A\cap (z_\lambda - B_\lambda) \neq \emptyset$. Hence, there are $a_\lambda \in A$ and $b_\lambda \in B_\lambda$ such that $z_\lambda - b_\lambda = a_\lambda$. Set $x_\lambda = b_\lambda$. It clarly satisfies (2) and (3).


Now, let us  define $X= \{x_\alpha \}_{\alpha < \mathfrak{c}}$. By (2), $A+X=2^{\omega}$. Hence, the set $X$ does not belong to $(\mathcal{J}\cup\{A\})^*$. 

To show that $X\in \mathcal{J}^{*}$ let us fix $J\in\mathcal {J}$. Notice that for some $\alpha<\mathfrak c$ $J\subseteq F_\alpha$. Take $\lambda>\alpha$. 
By (1), $r_\lambda\notin F_\alpha+\{x_\xi:\ \xi<\lambda\}$ and by (3), for $\xi\ge\lambda$, $x_\xi\notin r_\lambda-F_\alpha$. Hence, $r_\lambda\notin F_\alpha+\{x_\xi:\ \xi\ge\lambda\}$. It shows that $X\in \mathcal{J}^{*}$.
\end{proof}

Previous Theorem also holds for Polish groups of cardinality continuum.

\begin{cor} 
$$add(\mathcal{M})=\mathfrak{c}\Longrightarrow \mathcal{M}=\mathcal{M}^{**},$$
$$add(\mathcal{N})=\mathfrak{c}\Longrightarrow \mathcal{N}=\mathcal{N}^{**}.$$
\end{cor}

The implication in \Cref{bbbbb} cannot be reversed, since Pawlikowski and Sabok showed in \cite{Pawlikowski2008-PAWTS} that $Count = Count^{**}$ and $\add(Count) = \omega_1$.

Now we will show some examples of families where we have $\mathcal{F}=\mathcal{F}^*$. Let us start with the space of integers.

\begin{proposition}
    There exists a set $\mathcal{F}\subseteq\mathcal{P}(\mathbb{Z}) $ such that $\mathcal F=\mathcal{F}^{*}$.
\end{proposition}

\begin{proof}
    Take 
    $$\mathcal{F}=\{x: x\in \mathcal{P}({O})\vee x\in \mathcal{P}({E})\},$$
    where ${O}$ is the set of odd numbers and ${E}$ is the set of even numbers. 

    Take any $x,y\in \mathcal{F}$. If both $x$  and $y$ contain numbers of the same parity, then $x+y$ can only contain even numbers. If $x$ and $y$ are of different parity, then $x+y$ can contain only odd numbers.

    Take any $x\notin\mathcal{F}$, then $x$ contains numbers of different parity. Hence, $x+\mathcal{O}= \mathcal{O}\cup\mathcal{E}=\mathbb{Z}$. So $x\not\in \mathcal{F}^*$.
\end{proof}

The next example is built in the Cantor space.

\begin{proposition}\label{f=f*cantor}
    There exists $\mathcal{F}\subseteq \mathcal{P}(2^{\omega})$ such that $\mathcal{F}=\mathcal{F}^{*}$.
\end{proposition}

\begin{proof}
    Let  $\mathcal{U}$ be an ultrafilter on $\omega$. Take a set:
$$\mathcal{F}=\big\{A\subseteq2^{\omega}: \forall f,g\in A\; \{n:f(n)=g(n)\}\in\mathcal{U}\big\}.$$
To show that $\mathcal{F}\subseteq \mathcal{F}^*$ take any $A,B\in\mathcal{F}$. Notice that for every $f, f'\in A$ and for every $g, g'\in B$ 
$$
\{n:f(n)=f'(n)\}\cap\{n: g(n)=g'(n)\}\in\mathcal{U}.
$$
Hence 
$$
\forall f, f'\in A\;  \forall g, g'\in B \quad \{n:f(n)+g(n)=f'(n)+g'(n)\}\in\mathcal{U}
$$
and, therefore, $A+B\neq \mathcal{P}(2^\omega)$.

To show that $\mathcal{F}^*\subseteq \mathcal{F}$ take $B\notin \mathcal{F}$. Then there exists $f,g \in B$ such that: 
$$\{n:f(n)\neq g(n)\}\in\mathcal{U}.$$ 
Define $A=\big\{x\in 2^\omega:\ \{n:\ x(n)=0\}\in\mathcal{U}\big\}$. Notice that 
$$
A+B\supseteq A+\{f,g\}\supseteq \big\{x:\ \{n:\ x(n)=0\}\in\mathcal{U}\big\}\cup \big\{x:\ \{n:\ x(n)=1\}\in\mathcal{U}\big\}= 2^\omega.
$$
Hence, $B\notin \mathcal{F}^*$.
\end{proof}

\begin{proposition}

    There exists $\mathcal{F}\subseteq \mathcal{P}(\mathbb{Z}^{\omega})$ such that $\mathcal F= \mathcal F^{*}$.
\end{proposition}

\begin{proof}
    Let $\mathcal{U}$ be an ultrafilter on $\omega$. Set

$$\mathcal{F}=\big\{A\subseteq\mathbb{Z}^{\omega}: \forall f,g\in A\; \{n:2|f(n)-g(n)\}\in\mathcal{U}\big\}.$$
The rest of the proof is similar to the proof of Theorem \ref{f=f*cantor}.   
\end{proof}

In \cite{Kysiak}, M. Kysiak introduced the notion of very meager sets, motivated by the observation that strongly meager sets are not necessarily closed under countable, or even finite, unions. The class of very meager sets, however, forms a $\sigma$-ideal. The following definition provides a natural generalization of this concept.

\begin{df}
For a family $\mathcal{F}\subseteq \mathcal{P}(X)$, where $X$ is polish group let:
$$\hat{\mathcal{F}}=\{A\subseteq X:\; \forall F\in \mathcal{F}\; \exists T\in [X]^{\omega}\; A\subseteq F^c+T\}.$$
\end{df}
Of course we have 
$\mathcal{F}^{*}\subseteq \hat{\mathcal{F}}$.

It turns out that for some families 
the result of this operation is, unfortunately, the whole space. It is the case for $\sigma$-compact subsets of the Baire space.

\begin{thm}
    $\hat{K}_{\sigma}=\mathcal{P}(\mathbb{Z}^{\omega})$.
\end{thm}

\begin{proof}
    Take any $F\in K_\sigma$. Then there exists $f\in\omega^\omega$ such that $$F\subseteq\{x\in \mathbb{Z^{\omega}: \forall^{\infty}}n \; |x(n)|< f(n)\}$$ and the latter set is still $K_\sigma$. 
    Set 
    $T=\{a,b\}$ where $a=\{(n,0)\; n\in\omega\}$ and $b=\{(n,2f(n)+1):\; n\in\omega\}$. Then $F^c+\{a\}=F^c$ and $F^c+\{b\}\supseteq F$. Hence $F^c+T=\mathbb{Z}^\omega$.
\end{proof}
\begin{df}
    
We say that two ideals $\mathcal{I}$ and $\mathcal{J}$ on $2^\omega$ are orthogonal $\mathcal{I}\perp\mathcal{J}$ if there exist $I\in \mathcal{I}$ and $J\in \mathcal{J}$ such that $I\cup J=2^\omega$.
\end{df}

The most well-known examples of orthogonal ideals are the ideals of meager sets and null sets.

The next theorem tells us when $\hat{\mathcal{J}}$ is not the whole space.

\begin{thm}
    $\hat{\mathcal{J}}\neq \mathcal{P}(2^\omega) \Longleftrightarrow \text{ there exists a proper $\sigma$-ideal }  \mathcal{I}$ such that ${\mathcal{I}\perp \mathcal{J}}$.
\end{thm}

\begin{proof}
    $(\Longleftarrow)$ Let  $J\in \mathcal{J}$ and $I\in J$  be such that  $I\cup J=2^\omega$.  Then for any countable $T$ we have $J^c+T=I+T\in \mathcal{I}$. Hence, $J^c+T\neq2^{\omega}$.

    $(\Longrightarrow)$ Assume there doesn't exists ideal $\mathcal{I}$ such that $\mathcal{I}\perp \mathcal{J}$.
    
    Take $J\in \mathcal{J}$. Let $\mathcal{I}$ be
    the $\sigma$-ideal generated by translations of $J^c$, i.e. 
    $$\mathcal{I}=\{I\subseteq 2^\omega:\ \exists T\in[2^\omega]^\omega\; I\subseteq J^c+T\}.$$
    Clearly $\mathcal{I}\perp \mathcal{J}$. So $\mathcal{I}$ is not proper.
    Hence, there is a countable set $T$ such that $J^c+T=2^\omega$.
    It implies that $2^\omega\in \hat{\mathcal{J}}$.
\end{proof}

\section{Strong measure zero and strongly meager sets }

We begin with the classical definition of strong measure zero sets on the real line, originally introduced by Borel.

\begin{df} 
    A set $A \subseteq \mathbb{R}$ has strong zero measure if for every sequence $(\varepsilon_n)$ of positive reals there exists a sequence $(I_n)$ of intervals such that $|I_n|\leqslant \varepsilon_n$ and $A$ is contained in the union of $I_n$.
\end{df}

In \cite{galvin1973strong} Galvin, Mycielski and Solovay showed the relation between strong measure zero sets and meager sets.

\begin{thm}(Galvin-Mycielski-Solovay)
        A set $X\subseteq \mathbb{R}$ has strong measure zero if and only if for every meager set $H$ we have $X+H \neq \mathbb{R}$.
\end{thm}

We will work on the Cantor space so we translate the definition of strong measure zero sets from $\mathbb{R}$ to $2^\omega$.

\begin{df}
    A set $A \subseteq 2^\omega$ has strong zero measure $(A\in \mathcal{SMZ})$ when for every sequence $(k_n)$ of natural numbers there exists a sequence of clopens $(\sigma_n)$ such that $|\sigma_n|=k_n$ and $A\subseteq \bigcup [\sigma_n]$ .
\end{df}

In his PhD Thesis "Special sets of real numbers and variants of the Borel Conjecture"  \cite{Wohofsky}, Wohofsky proved that the Galvin-Mycielski-Solovay Theorem holds in $2^\kappa$ where $\kappa$ is weakly compact and is equipped with the topology generated by the collection $\left\{[s]: s \in 2^{<\kappa}\right\}$.

Also in his thesis, he showed that the Galvin-Mycielski-Solovay Theorem doesn't hold in the Baire space.

To prove the Galvin-Mycielski-Solovay Theorem in the Cantor space, we will need the following Lemma.

\begin{lm} \label{XD}
    For every nowhere dense set $C$ we have: $$\forall m \exists k (\forall \alpha\in2^m)( \forall \beta\in 2^{m+k})( \exists\gamma \in 2^{m+k}) (\alpha\subseteq \gamma) (([\gamma]+[\beta])\cap C =\emptyset).$$
\end{lm}

\begin{proof}
    Take any level $m.$ We have there $2^m$  of $\alpha$ and for every $\alpha$ we have $k$ such that there exists $\delta \in 2^{m+k}$, $\alpha \subseteq \delta$ and $[\delta]\cap C=\emptyset$. Take biggest of those $k.$ 

    Take any  $\alpha \in 2^m$ and any $\beta\in 2^{m+k}$. Take $\alpha+\beta\upharpoonright m$. There exists  $\delta$ such that $ \alpha+\beta\upharpoonright m\subseteq\delta$, $|\delta|=m+k$ and $[\delta]\cap C=\emptyset.$ Let $\gamma$ be such that $\beta +\gamma =\delta.$  
\end{proof}


\begin{thm}(Galvin-Mycielski-Solovay)  
           A set $X\subseteq 2^\omega$ has strong measure zero if and only if for every meager set $H$ we have $X+H \neq 2^\omega$. 
\end{thm}
\begin{proof}
$(\Longleftarrow)$ Suppose $X\subseteq 2^\omega$ is such that for every meager set $H$ we have $X+H\neq 2^\omega$. To show that $X\in \mathcal{SMZ}$ take any sequence $(k_n: n\in\omega)$. Let $(t_n)$ be such that  $|t_n|=k_n$ and $O=\bigcup[t_n]$ is a dense subset of $2^\omega$. Set $C=2^\omega \backslash O.$ Since $X+C\neq 2^\omega$ there exists $z\in 2^\omega$ such that $X\cap (C+z)=\emptyset$. So $X\subseteq O+z$. Hence, $X$ has strong measure zero.

$(\Longrightarrow)$ Take any $H \in \mathcal{M}$, then $H\subseteq \bigcup C_n$, where $C_n$ are nowhere dense sets.

Apply \Cref{XD} for $m=0$ and $C_0$, we get $k_0$.
Now apply \Cref{XD} for $m=k_0$ and $C_1$, we get $k_1$.  When we continue with this procedure, we get a sequence of $k_n$.

Since $X\in \mathcal{SMZ}$,   $X\subseteq\bigcup[\sigma_n]$, where $|\sigma_n|=k_n.$ 

For $\sigma_0$ we take $\gamma_0$ for $C_0$ from \Cref{XD} such that $([\gamma_0]+[\sigma_0])\cap C_0 =\emptyset$.  For $\sigma_1$ apply \Cref{XD} for $\alpha=\gamma_0$ and $\beta=\sigma_1$ we get $\gamma_1$ such that $([\gamma_1]+[\sigma_1])\cap c_1=\emptyset.$ In $n+1-th$ step apply \Cref{XD} for $\alpha=\gamma_{n-1}$ and $\beta=\sigma_n$. We get $([\gamma_n]+[\sigma_n])\cap C_n =\emptyset.$

Let $y=\bigcap [\gamma_n]$, then $(y+X)\cap H=\emptyset$.
\end{proof}

This Theorem naturally leads to another definition of strongly meager sets defined by Prikry.

\begin{df}
A set $X \subseteq 2^{\omega}$ is strongly meager $(X\in \mathcal{SM})$ if for every measure zero set $H$ it holds that $X+H \neq 2^\omega$.
\end{df}

Using operation $^*$, we have 
$$\mathcal{SM}=\mathcal{N}^*,\quad \mathcal{SMZ}=\mathcal{M}^*.$$
 
In 1919, Borel conjectured that all strong measure zero sets are countable. It was later shown by Laver in \cite{BC} that this conjecture is consistent with the axioms of ZFC.

Similarly, the dual Borel Conjecture states that all strongly meager sets are countable. It was shown by Carlson in \cite{dBC} that this conjecture is also consistent with the axioms of ZFC.

In 2011 Martin Goldstern, Jakob Kellner, Saharon Shelah and Wolfgang Wohofsky showed  in \cite{BCBC} that it is consistent that the Borel Conjecture and the dual Borel Conjecture hold simultaneously.

\begin{thm}
Borel Conjecture  $\Longrightarrow \mathcal{M} \neq \mathcal{M}^{**}$.
\end{thm}

\begin{proof}
From the Borel Conjecture, we know that $\mathcal{SMZ} = Count$. So, $\mathcal{M}^{*} =Count$, which implies that $\mathcal{M}^{**} = Count^{*}$. Since $\mathcal{M} \neq Count^{*}$, we obtain that $\mathcal{M} \neq \mathcal{M}^{**}$
\end{proof}
Similarly, we know that the dual Borel Conjecture $\Longrightarrow \mathcal{N} \neq \mathcal{N}^{**}$.

\section{Porous sets and microscopic sets}

The modern notion of porosity was formulated by J.Väisälä in \cite{PpP}.
He proved that porosity is preserved under quasisymmetric mappings and that every porous set in 
$\mathbb{R}^n$  has Hausdorff dimension strictly less than $n$.

\begin{df}
    We say that $X\subseteq \mathbb{R}$ is porous $(X\in \mathcal{P})$ if there exists $\alpha \in (0,1)$ for all $x\in \mathbb{R}$ and any $\varepsilon$ there exists $y\in \mathbb{R}$ such that $(y-\alpha \varepsilon, y+\alpha \varepsilon)\subseteq (x-\varepsilon, x+\varepsilon) \backslash X. $
\end{df}

In \cite{hruvsak2012cardinal}, Zindulka and Hru{\v{s}}{\'a}k extended the notion of porous sets to the Cantor space.

\begin{df}
We say that $X\subseteq 2^{\omega}$ is porous $(X\in \mathcal{P})$ if there exist $k$ for every $m\in \omega$ and  $\alpha \in 2^{m}$, there exists $\beta \in 2^{m+k}$, $\alpha \subseteq \beta$ and $[\beta]\cap X=\emptyset.$   
  \end{df}

  \begin{df}
We say that $X\subseteq 2^{\omega}$ is $k$-porous $(X\in k\mathcal{P})$  for every $m\in \omega$ and $\alpha \in 2^{m}$ there exists $\beta \in 2^{m+k}$, $\alpha \subseteq \beta$ and $[\beta]\cap X=\emptyset.$   
  \end{df}

  \begin{df}
      We say that $X\subseteq 2^{\omega}$ is $\sigma$-porous $(X\in \sigma\mathcal{P})$ if it is a countable union of porous sets.
  \end{df}
\begin{thm}(Folklore) 
        Every porous set has measure zero.

\end{thm}

\begin{proof}
We will make this proof on the interval $[0,1]$.  
Take porous set $X$ and take $\alpha$ associated with this set. Since it is porous, there exists $I_{\sigma_0}\subseteq[0,1]$ such that $|I_{\sigma_0}|=\alpha$ and $I_{\sigma_0}\cap X=\emptyset$. Next in $[0,1]\backslash I$ there exists $I_{\sigma_0},I_{\sigma_1}\subseteq [0,1]\backslash I_{\sigma_0}$ such that $|I_{\sigma_0}|+|I_{\sigma_1}|=(1-\alpha)\alpha$ and $I_{\sigma_0}\cap X=\emptyset$ and $I_{\sigma_1}\cap X=\emptyset$.

We continue with this procedure, and in the $n-th$ step we have set of intervals of total length of $(1-\alpha)^n $ and we remove from it intervals $I_{\sigma_0}, ..., I_{\sigma_{2^n}}$ of total length of $(1-\alpha)^{n}\alpha$ such that for every $\sigma_n$ we have $I_{\sigma_n}\cap X=\emptyset$. 

The sum of the lengths of the intervals we removed from $[0,1]$ is of the form:

$$\alpha \sum_{k=0}  (1-\alpha)^k = \frac{\alpha}{1-(1-\alpha)}=1.$$
\end{proof}
Of course, every porous set is nowhere dense. Moreover, since every porous set is contained in a closed porous set, the conclusions follow.

$$\sigma\mathcal{P}\subseteq \mathcal{E}\subseteq \mathcal{M}\cap \mathcal{N}\subseteq\mathcal{N}$$

The concept of upper porous sets was first introduced by Dolzhenko in 
\cite{porowate}, where he established the fundamental properties of these sets.

The notion was first introduced for metric spaces, but here we recall the definition for the real line.

\begin{df}
      We say that $E\subseteq \mathbb{R}$ is upper porous $(X \in \mathcal{P}_{u})$ if for every $x\in E$
            $$\underset{\varepsilon\rightarrow 0^+}{lim sup} \frac{\sup\{b-a:\; (a,b)\subseteq (x-\varepsilon, x+\varepsilon)\setminus E\}}{\varepsilon}>0.$$

\end{df}
In \cite{Nowik} Klinga, Nowik and Wąsik extended the notion of upper porous sets to the Cantor space.

\begin{df}
We say that $X\subseteq 2^{\omega}$ is upper porous $(X\in \mathcal{P}_{u})$ if for all $x\in X$ there exists $K$ such that for infinitely many $n$ there exist $\beta\supseteq x\upharpoonright n$, $|\beta|=n+K$, and $[\beta]\cap X=\emptyset$.   
  \end{df}
  \begin{df}
      We say that $X\subseteq 2^{\omega}$ is $\sigma$-upper porous $(X\in \sigma\mathcal{P}_u)$ if it is a countable union of upper porous sets.
  \end{df}
\begin{df}
    Let $A \subseteq \mathbb{R}$ be a measurable set. 
A point $x \in \mathbb{R}$ is called a point of density $1$ of $A$ if
\[
\lim_{r \to 0} \frac{\lambda\big(A \cap B(x,r)\big)}{\lambda\big(B(x,r)\big)} = 1,
\]

\end{df}
\begin{thm}
    Every $X\in\mathcal{P}_{u}$ set has measure zero.
\end{thm}

\begin{proof}
    Assume there exists a set $A$ that is upper porous and $\lambda(A)>0$. Because $A$ has a positive measure, then there exists a point $a\in A$ of density $1$. But then $\underset{\varepsilon\rightarrow 0^+}{lim sup} \frac{\sup\{b-a:\; (a,b)\subseteq (x-\varepsilon, x+\varepsilon)\setminus E\}}{\varepsilon}=0$, so $A$ cannot be upper porous.    
\end{proof}

We have the following conclusion:

$$\sigma\mathcal{P}\subseteq\sigma\mathcal{P}_{u}\subseteq \mathcal{E}\subseteq \mathcal{M}\cap \mathcal{N}\subseteq\mathcal{N}$$

An example of a set that is upper porous but not $\sigma$-porous was given 
independently by Klinga, Nowik, and Wąsik in \cite{Nowik}; however, our example differs. 

\begin{thm}
There exists $X\in \mathcal{P}_{u}$ such that $X\notin \sigma\mathcal{P}.$
\end{thm}

\begin{proof}
Take $X=\{y\in 2^{\omega}:\ \forall n\; y(2^n)=y(2^n +1)=0\}$.

   First we will show that   $X\in \mathcal{P}_{u}$.

    Take any $y\in X$ and let $K=2$. Then for $y\upharpoonright 2^n -1$ and for every $n$ let $\beta=(y\upharpoonright2^n -1 )^\frown1^\frown1$. Of course $[\beta]\cap X =\emptyset$ and we can construct such $\beta$ for every $n$, so $X\in \mathcal{P}_{u}$. 

    Next step is showing that $X\notin \sigma\mathcal{P}.$

     Define $P\in \sigma\mathcal{P}$. We will show that $X\backslash P\neq\emptyset .$ Assume that $P=\bigcup_{j\in \omega} P_j$ where $P_j$ is $k_j$-porous. Without loss of generality the sequence $(k_j)_{j\in\omega}$ is increasing.

     We want to construct sequence $(\tau_n:n\in\omega)$ satisfying following conditions:
     
     \begin{enumerate}
     
     \item $\tau_{n} \subseteq \tau_{n+1}$,
     \item $[\tau_n] \cap P_n =\emptyset$,

         \item if $b\in\{2^a,2^{a}+1: a\in \omega \}\cap dom(\tau_n)$ then $\tau_n (b)=0$.
         
            \end{enumerate} Let $m_{k_j}$ be smallest $m$ such that $k_j<2^m -2$.
     
     Start with $\alpha_0=\underbrace {0^\frown 0^\frown \cdots ^\frown0}_{2^{m_{k_0}}+1}$. Next, take $\tau_0 \supseteq \alpha_0$, $|\tau_0|=|\alpha_0|+k_0$ so that $[\tau_0] \cap P_0=\emptyset$.

    Next in the $n-th$ step, let $\alpha_n= {\tau_{n-1}}^\frown\underbrace {0^\frown 0^\frown \cdots ^\frown0}_{2^{m_{k_n}}+1 -(2^{m_{k_{n-1}}}+1)}$.  Next, we take $|\tau_n|=|\alpha_n|+k_n$ so that $[\tau_n]\cap P_n =\emptyset$. 
    
   Set  $t=\bigcup_{n\in\omega} \tau_n$. Notice that $t \in X$, because of  $(3)$ and          $t \notin \bigcup P_j$ because of $(2).$
\end{proof}

Next, we turn our attention to another ideal: microscopic sets. This concept was studied, for instance, in \cite{KWELA201651}, where Kwela showed that the additivity of the ideal of microscopic sets is always equal to $\omega_1$. In \cite{unknown} O.Zindulka  investigated Hausdorff dimension of microscopic sets in $\mathbb{R}^d$. In \cite{M} E.Wagner-Bojakowska A.Karasińska A.Paszkiewicz introduced some generalizations of microscopic sets.

\begin{df}
       We say that a set $X \subseteq \mathbb{R}$ is microscopic $(X\in Micro)$ if for all $\varepsilon\in \mathbb{R}$ there exists a sequence of intervals $(I_n)$, such that $|I_n|=\varepsilon^{n}$ and $X\subseteq \bigcup I_n$.
\end{df}

We will consider microscopic sets on Cantor space.

\begin{df}
  In $2^{\omega}$ we say that a set $X \subseteq 2^{\omega} $ is microscopic $(X\in Micro)$ if for all $k\in \mathbb{N}$ there exists a sequence of $(\sigma_n)$ such that $|\sigma_n|=k\times n$, and $X\subseteq \bigcup [\sigma_n]$.
  \end{df}

It turns out that the Galvin–Mycielski–Solovay theorem almost extends to microscopic and porous sets, but only one direction of the implication remains valid.

The following lemma is crucial for establishing one direction of the Galvin–Mycielski– Solovay theorem in the case of porous and microscopic sets.

\begin{lm} \label{aa}
     If $E\in \mathcal{P}$, then there exists $k\in \mathbb{N}$ such that for all $m\in \mathbb{N}$, $\alpha \in 2^m$, $\tau \in 2^{m+k}$, there exists $\beta \in 2^{m+k}$ such that $\alpha \subseteq \beta$ and $([\tau] + [\beta]) \cap E =\emptyset$.
\end{lm}

\begin{proof}
    Take any $E\in \mathcal{P }$, $\alpha \in 2^{m}$, and $\tau \in 2^{m+k}$ where $k$ is the same as in the set $E$.  Take $\alpha + \tau \upharpoonright n =\gamma$. Since $E$ is porous, there exists $\delta$ such that $|\delta|=m+k$, $\gamma\subseteq \delta$, and $[\delta] \cap E=\emptyset$. Take $\beta$ such that $\alpha \subseteq \beta $ and $\tau +\beta =\delta$.
\end{proof}

Let us remark that an analogous result to the following theorem was announced by O. Zindulka at the 2022 Summer Symposium in Real Analysis in Paris, where it was stated on the real line.
More general theorem
can be found in work from 2026 \cite{Zindulka2026MoreOD}, see Theorems
5.3 and 5.5.

\begin{thm}
        For any $X\in Micro$ and any $E\in\mathcal{P}$ we have $X+E\neq 2^\omega$.
        

\end{thm}

\begin{proof}

Proof of this Theorem is similar to the proof of the Galvin-Mycielski-Solovay theorem on the Cantor space.

    Let $E\in \mathcal{P}$ and take $k$ from \Cref{aa}. Take any $X\in Micro$ and its sequence of $(\sigma_n)$ such that $|\sigma_n|=(n+1)k$ and $X\subseteq \bigcup [\sigma_n]$.
    We want to construct a sequence $(\alpha_n: n<\omega)$ satisfying following conditions:

    \begin{enumerate}
        \item $\alpha_n\subseteq\alpha_{n+1}$,
        \item $[\alpha_n] + [\sigma_n]\cap E =\emptyset$.
    \end{enumerate}

    Apply \Cref{aa} for $\alpha\in 2^0$ and take $\sigma_0$, then there exists $\beta_0 \in 2^k$ such that $[\sigma_0]+[\beta_0]\cap E=\emptyset$ and $\alpha\subseteq\beta$. Take $\alpha_0=\beta_0$.
    
    In $n$-th step apply \Cref{aa} to $\alpha_{n-2}$ and for $\sigma_{n-1}$. Then there exists $\beta_{n-1}\in 2^{nk}$ such that $[\sigma_{n-1}]+ [\beta_{n-1}] \cap E =\emptyset$. Take $\alpha_{n-1}=\beta_{n-1}.$
    
    In the end, we take $y=\bigcup \alpha_n$. Of course, we have $(y+X)\cap E=\emptyset$ because of $(2)$.
\end{proof}

In terms of operation $^{*}$ we get:
$$Micro\subseteq\mathcal{P}^{*}, \quad \mathcal{P}\subseteq Micro^{*}.$$

We will use the following Theorem to show that the converse of those inclusions does not hold.

\begin{thm}\label{dddddddd}
    For $M\in Micro$ and $k>0$ there exists  a sequence of clopens $(\sigma_n)_{n\in\omega}$,  $|\sigma_n|=kn$ such that
    $$
    y\in M\Longrightarrow \exists^{\infty}n\; \sigma_n \subseteq y
    $$

    for every $y\in 2^\omega$.   
\end{thm}
\begin{proof}
    Take any $k.$ Since $M$ is microscopic we can cover it with a sequence $(\tau_n^2)$, $|\tau_n^2|=2nk$.  Take $\sigma_{2n}=\tau_n^2$. 
    
    We have also $M\subseteq \bigcup[\tau_n^4]$ where $|\tau_n^4|=4nk.$ Take $\sigma_{4n-1}=\tau_n^4 \upharpoonright (4n-1).$
    
    Next we can take $(\tau_n^8)$ where $|\tau_n^8|=8nk$ and $M\subseteq \bigcup[\tau_n^8]$. Take $\sigma_{8n-3}=\tau_n^8\upharpoonright (8n-3)$.

    We continue with this and get the sequence $(\sigma_n).$
    \end{proof}

In \cite{EEEE} there is a simpler characterization of closed null sets, then we can find in \cite{BBB}.

Let $I_n$ be an interval partition of $\omega$ and $S_n \subseteq 2^{I_n}$ such that $\frac{|S_n|}{|I_n|}\leq \frac{1}{2}$, then for every $E\in \mathcal{E}$ we have:
$$E\subseteq \{x:\forall^{\infty}_n x \upharpoonright I_n \in S_n\}.$$
The idea given by this characterization is useful in constructing the set in the following Theorems.

\begin{thm}
    $Micro^* \neq \sigma\mathcal{P}$.
\end{thm}
\begin{proof}
    Let $I_n$ be a sequence of consecutive intervals on $\omega$ such that $|I_n|=2n$. Let $J_n\subseteq I_n$ be a subinterval of length $n$, i.e.
   $$I_n=[n(n-1), (n+1)n -1]\text{ and }J_n=[n(n-1), (n-1)(n+1)].$$    
    Set $$E=\{x\in 2^\omega: \forall^{\infty}n \; x\upharpoonright (I_n \backslash J_n)=0 \}.$$

    \begin{claim*}
    $E$ is not $\sigma$-porous.
        
    \end{claim*}

    \begin{proof}
     Take $P\in \sigma\mathcal{P}$. Assume that  $P=\bigcup_{j\in \omega} P_j$ where $P_j$ is $k_j$-porous. Without loss of generality the sequence  $(k_j:\ j\in\omega)$ is increasing. 
     
     Recall that for every $j\in\omega$ and every $\alpha\in 2^{<\omega}$ there exists $\tau^\alpha_j \in 2^{k_j}$ such that $[\alpha^\frown\tau_j] \cap P_j =\emptyset$.

     We want to construct sequence $(\tau_n:n\in\omega)$ satisfying following conditions:
          \begin{enumerate}
   \item $\tau_{n} \subseteq \tau_{n+1}$,
     \item $[\tau_n] \cap P_n =\emptyset$,

         \item if $b\in\bigcup_{l\in\omega}(I_l\backslash J_l) \cap dom(\tau_n)$ then $\tau_n (b)=0$.

         
            \end{enumerate}
     Start with $\alpha_0=\underbrace {0^\frown 0^\frown \cdots ^\frown0}_{(k_0 ^2)}$. Next, take $\tau_0 \supseteq \alpha_0$, $|\tau_0|=|\alpha_0|+k_0$ such that $[\tau_0] \cap P_0=\emptyset$.

    Next in the $n$-th step, let $\alpha_n= \tau_{n-1}^\frown\underbrace {0^\frown 0^\frown \cdots ^\frown0}_{k_n^2 - k_{n-1}^2-k_{n-1}}$.  Next, we take $\tau_n\supseteq\alpha_n$, $|\tau_n|=|\alpha_n|+k_n$ such that $[\tau_n]\cap P_n =\emptyset$.

    Finally, set $t=\bigcup_{n\in\omega}\tau_n$. We get that $t\notin P$, because of $(2)$. Moreover, we obtain that $t \in E$, because of $(3)$.
\end{proof}

\begin{claim*}
    For every $M\in Micro$, it holds that $M+E\neq 2^\omega$.
    
\end{claim*}
\begin{proof}

We will use the characterization form \Cref{dddddddd}.

    Fix $M\in Micro$. Then there exists $(\sigma_m)_{m<\omega}$ such that $\sigma_m \in 2^{3m}$ and
     $$M\subseteq \{x: \exists^{\infty }m\; x\upharpoonright3m=\sigma_m\}=\hat{M}$$
    

    We will define $z\notin E+\hat{M}$. For $i\in \bigcup_{n=0}^\infty J_n$, set $z(i)=0$.

   Define $h_k\in (I_{j+1}\backslash J_{j+1})$ where 
  $$h_k=\left(j+1\right)^2 +\left(k - \frac{j\left(j+1\right)}{2}\right) -1$$
   and $j$ is maximal such that $\frac{j(j+1)}{2}<k$. 

Note that $h_k$ is the first place where for $y\in E $, $y(h_k)$ can be something different than $0$.

Define $z(h_k)$ such that $z(h_k)\neq\sigma_k(h_k).$ For all $m$ there exists $n$ and $h\in (I_n \backslash J_n) \cap 3m $ we get $z(h)\neq \sigma_ m(h)+0$.  Notice that $z\notin E+\hat{M}$.\end{proof}\end{proof}

Let us remark that the consistent example to the next Theorem is Sierpiński set. It was given by Arturo Martinez Celis at Wrocław seminar.

\begin{thm}
    $Micro\neq \sigma\mathcal{P}^{*}$.
\end{thm}

\begin{proof}

    Let us define $E=\{y:\; (\forall m)(\forall i<3^m)\;  y(3^m -1+i)=0\}$.

    \begin{claim*}

        $E\notin Micro$.
    \end{claim*}

\begin{proof}

    Fix $M\in Micro$. Then by characterization from \Cref{dddddddd} there exists $(\sigma_m)_{m<\omega}$ such that $\sigma_m \in 2^{5m}$ and $M\subseteq \{x: \exists^{\infty }m\; x\upharpoonright5m=\sigma_m\}=\hat{M}$.

Let $Z=\{n: \exists m \exists i<3^m  n=3^m-1+i\}$. Take set $\omega\backslash Z=\{k_n: n\in\omega \}$ where the sequence $(k_n  )_{n\in\omega}$ is increasing. Notice that $k_n<5n$. Indeed:
\begin{enumerate}
    \item if $2\cdot 3^k -1 \leq 5n < 3^{k+1}$
    
We have
    $|Z \cap 5n|= 1+..+3^k = \frac{3^{k+1} -1}{2}$. 
    
   To show that $k_n<5n$ we have to prove that $4n>\frac{3^{k+1}-1}{2}$. We get $8n>3^{k+1}-1$ which is true, because  $2\cdot 3^k -1 \leq 5n$.

    \item if $2\cdot 3^k -1 >5n \geq 3^k$
    
We have
    $|Z \cap 5n|= 1+..+3^{k-1} +( 5n -3^k +1)$.
    
    To show that $k_n<5n$  we have to prove that $4n>\frac{3^{k}-1}{2} +5n -3^k +1$. We get $n<3^k -1 -\frac{3^k -1}{2}$. And $5n<5\cdot \frac{3^k}{2} -\frac{5}{2}$ which is true, because  $2\cdot 3^k -1 > 5n$.

\end{enumerate}

 We want to construct $z\in 2^\omega$ satisfying following conditions:

     \begin{enumerate}

   \item $z(n)=0$ if $n\in Z$,
     \item $z(k_n)\neq \sigma_n (k_n)$. 

     \end{enumerate}






    Because of $(1)$ we have that $z\in E$ and $z\notin \hat{M}$ because of $(2)$. 
\end{proof}

\begin{claim*}

    For any $X\in \sigma\mathcal{P}$ we have $X+E\neq 2^\omega$.
    
\end{claim*}

\begin{proof}
     Let $X=\bigcup P_k$, where $P_k $ is $3^{p_k}$-porous and $p_k$ is increasing. We want to construct sequence $(\tau_n:n\in\omega)$ satisfying following conditions:

     \begin{enumerate}

   \item $\tau_{n} \subseteq \tau_{n+1}$,
     \item $[\tau_n] \cap P_n =\emptyset$,

         \item $\tau_n(x)=0$ if $x\neq3^m-1+i$ for some $m$ and $i<3^m$.

     \end{enumerate}

  First we take $\alpha_0= {\underbrace {0^\frown 0^\frown  \ldots^\frown 0}_{2\cdot 3^{k_0}-1 }}  $. Take $\tau_0 \supseteq \alpha_0$, $|\tau_0|=|\alpha_0|+3^{k_0}$ such that $[\tau_0] \cap P_0=\emptyset$.

  Next, take $\alpha_1 =\tau_0 ^\frown{\underbrace {  0 ^\frown 0^\frown\ldots^\frown 0}_{2\cdot3^{k_1} - 3\cdot 3^{k_0}}}$. Take $\tau _1 $, $\tau_1 \supseteq \alpha_1$, $|\tau_1|=|\alpha_1|+3^{k_1}$ such that  $[\tau_1]\cap P_1=\emptyset$.

  In $n$-th step let $\alpha_n =\tau_{n-1} ^\frown{\underbrace{ 0^\frown 0^\frown\ldots^\frown 0}_{2\cdot 3^{k_n} -3\cdot 3^{k_{n-1}}}}$. Take $\tau_n$, $\tau_n \supseteq \alpha_n$, $|\tau_n|=|\alpha_n|+3^{k_n}$ such that $[\tau_n]\cap P_n = \emptyset$.

    Finally, take set $t=\bigcup_{n\in\omega}\tau_n$. We get that $t\notin X+E$, because of points $(2)$ and 
    $(3)$.\end{proof}\end{proof}

\printbibliography

\end{document}